\documentclass[12pt]{article}
\usepackage{amsmath}
\usepackage{amssymb}
\usepackage{latexsym}
\usepackage{amsthm}
\usepackage{amsmath,amssymb,amsfonts,amsthm}
\usepackage{mdframed}
\usepackage{graphicx}
\usepackage{stmaryrd}
\usepackage{multirow}
\usepackage{tablefootnote}

\usepackage{mathdots}
\usepackage{moresize}

\usepackage{stmaryrd}
\usepackage{multirow}
\usepackage{multicol}

\usepackage{graphicx}
\usepackage{tikz}
\usetikzlibrary{matrix,arrows}
\usetikzlibrary{positioning}
\usetikzlibrary{fit}
\usetikzlibrary{patterns}

\usepackage[colorlinks,
linkcolor=blue,
anchorcolor=blue,
citecolor=blue
]{hyperref}

\newtheorem{thm}{Theorem}[section]

\newtheorem{lem}[thm]{Lemma}

\newtheorem{coro}[thm]{Corollary}

\newtheorem{remark}[thm]{Remark}

\usepackage{graphicx}
\usepackage{tikz}
\usetikzlibrary{matrix,arrows}
\usetikzlibrary{positioning}
\usetikzlibrary{fit}
\usetikzlibrary{patterns}

\newcommand{\dbrac}[1]{{\llbracket #1 \rrbracket}} 	
\newcommand{\boks}[2]{({#1, #2})}   

\newcommand{\pattern}[4]{										
	\raisebox{0.6ex}{
		\begin{tikzpicture}[scale=0.35, baseline=(current bounding box.center), #1]
		\foreach \x/\y in {#4}		\fill[gray!20] (\x,\y) rectangle +(1,1);
		\draw (0.01,0.01) grid (#2+0.99,#2+0.99);
		\foreach \x/\y in {#3}		\filldraw (\x,\y) circle (5pt);
		\end{tikzpicture}}
}

\tikzset{global scale/.style={
		scale=#1,
		every node/.append style={scale=#1}
	}
}

\begin{document}
\begin{center}
{\large \bf  Distributions of several infinite families of mesh patterns}
\end{center}

\begin{center}
Sergey Kitaev$^{a}$,  Philip B. Zhang$^{b}$,  Xutong Zhang$^{c}$\\[6pt]

$^{a}$Department of Computer and Information Sciences \\
University of Strathclyde, 26 Richmond Street, Glasgow G1 1XH, UK\\[6pt]

$^{b,c}$College of Mathematical Science \\
Tianjin Normal University, Tianjin  300387, China\\[6pt]

Email:  $^{a}${\tt sergey.kitaev@cis.strath.ac.uk},
           $^{b}${\tt zhang@tjnu.edu.cn},
           $^{c}${\tt zhang.xutong@foxmail.com}
\end{center}

\noindent\textbf{Abstract.}
Br\"and\'en and Claesson  introduced mesh patterns to provide explicit expansions for certain permutation statistics as linear combinations of (classical) permutation patterns. The first systematic study of the avoidance of mesh patterns was conducted by Hilmarsson et al., while the first systematic study of the distribution of mesh patterns was conducted by the first two authors.

In this paper, we provide far-reaching generalizations for 8 known distribution results and 5 known avoidance results related to mesh patterns by giving distribution or avoidance formulas for certain infinite families of mesh patterns in terms of distribution or avoidance formulas for smaller patterns. Moreover, as a corollary to a general result, we find the distribution of one more mesh pattern of length 2. \\[5pt]

\section{Introduction}\label{intro}
Patterns in permutations and words have attracted much attention in the literature (see~\cite{Kit} and references therein), and this area of research continues to grow rapidly.
The notion of a {\em mesh pattern}, generalizing several classes of patterns, was introduced by Br\"and\'en and Claesson \cite{BrCl} to provide explicit expansions for certain permutation statistics as, possibly infinite, linear combinations of (classical) permutation patterns. Several papers are dedicated to the study of mesh patterns and their generalizations  \cite{AKV,Borie,JKR,KL,KR1,KRT,T1,T2}.

Let $\dbrac{0,k}$ denote the interval of the integers from $0$ to $k$. A pair $(\tau,R)$, where $\tau$ is a permutation of length $k$ written in one-line notation and $R$ is a subset of $\dbrac{0,k} \times \dbrac{0,k}$, is a
\emph{mesh pattern} of length $k$. Let $\boks{i}{j}$ denote the box whose corners have coordinates $(i,j), (i,j+1),
(i+1,j)$, and $(i+1,j+1)$. Let the horizontal lines represent the values,  and the vertical lines denote the positions in the pattern. Mesh patterns can be drawn by shading the boxes in $R$. For example, the picture
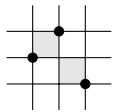
\[
\pattern{scale=1}{3}{1/2,2/3,3/1}{1/2, 2/1}
\]
represents the mesh pattern with $\tau=231$ and $R = \{\boks{1}{2},\boks{2}{1}\}$. A mesh pattern $(\tau,R)$ of length $k\geq 2$ is {\em irreducible} if the permutation $\tau=\tau_1\tau_2\cdots\tau_k$ is irreducible, that is, if there exists no $i$,  where $2\leq i\leq k$, such that $\tau_j<\tau_i$ for all $1\leq j<i$. For convenience, in this paper we assume that if  $\tau$ is of length 1 then it is {\em not} irreducible, even though normally such a $\tau$ is assumed to be irreducible.  All mesh patterns of interest in this paper can be found in Tables~\ref{tab-1}--\ref{tab-3}, where patterns' numbers $<66$ are coming from \cite{Hilmarsson2015Wilf,SZ}. Also, we let $Z:= \pattern{scale = 0.8}{1}{1/1}{0/0,1/1}$.

A subsequence $\pi'=\pi_{i_1}\pi_{i_2}\cdots\pi_{i_k}$ of a permutation $\pi=\pi_1\pi_2\cdots \pi_n$ is an occurrence of a mesh pattern $(\tau,R)$ if (a) $\pi'$ is order-isomorphic to $\tau$, and (b) the shaded squares given by $R$ do not contain any elements of $\pi$ not appearing in $\pi'$. For example, the mesh pattern of length 3 drawn above appears twice in the permutation $24531$ (as the subsequences 241 and 453). Note that even though the subsequences 251 and 451 are order isomorphic to 231 (the $\tau$ in the drawn pattern), they are not occurrences of the pattern because of the elements 4 and 3, respectively, be in the shaded squares. See \cite{Hilmarsson2015Wilf} for more  examples of occurrences of mesh patterns in permutations.

Let $S_n$ be the set of  permutations of length $n$.
Given a permutation $\pi$, denote by $p(\pi)$ the number of occurrences of pattern $p$ in $\pi$. Denote by  $S(p)$ set of  permutations avoiding $p$.
We let $A_p(x)$ be the generating function for $S(p)$ and let
$$F_p(x,q):=\sum_{n\geq 0}x^n\sum_{\pi\in S_n}q^{p(\pi)}.$$

In this paper we provide various generalizations of results in~\cite{SZ}. The main idea of this paper is to consider a mesh pattern $p$, and to replace some of its unshaded boxes by mesh patterns. To illustrate this idea, consider the mesh pattern

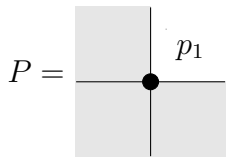
\begin{center}
$P=$ \begin{tikzpicture}[scale=1, baseline=(current bounding box.center)]
	\foreach \x/\y in {0/0,0/1,1/0}		
		\fill[gray!20] (\x,\y) rectangle +(1,1);
	\draw (0.01,0.01) grid (1+0.99,1+0.99);
	\filldraw (1,1) circle (3pt) node[below left]{\quad};
	\filldraw (1.2,1.7) circle (0pt) node[below right] {$p_1$};
	\end{tikzpicture}
\end{center}
obtained from the smaller mesh pattern $Y = \pattern{global scale = 0.8}{1}{1/1}{0/0,0/1,1/0}$ by inserting a mesh pattern $p_1$ in the box $(1,1)$. We can then find the distribution of $P$ in terms of the distribution of $p_1$, which not only allows us to obtain three results in~\cite{SZ} at the same time (distributions of the patterns Nr.\ 12, 13, and 17; see Section~\ref{sec2}) but also to derive the previously unknown distribution of the pattern
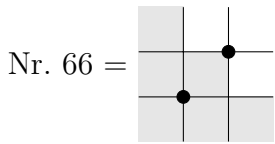
\begin{center}
	Nr.\ 66\ =\ \begin{tikzpicture}[scale=0.6, baseline=(current bounding box.center)]
	\foreach \x/\y in {0/0,0/1,2/0,1/0,0/2,1/1}		
	\fill[gray!20] (\x,\y) rectangle +(1,1);
	\draw (0.01,0.01) grid (2+0.99,2+0.99);
	\filldraw (1,1) circle (4pt) ;
	\filldraw (2,2) circle (4pt);
	\end{tikzpicture}
\end{center}
(pattern's number is introduced by us in this paper)   which is not equivalent to any of the patterns in~\cite{Hilmarsson2015Wilf}.

\begin{table}[htbp]
 	{
 		\renewcommand{\arraystretch}{1.7}
 \begin{center}
 		\begin{tabular}{|c|c|c|c|c|c|c|}
            \hline 			
 			{Nr.\ } & {Repr.\ $p$}  & {Generalization}  &  {\ $p_1$ } & {\ $p_2$}  & {\ $p_3$}  & {Distribution}
 			\\[8pt]
 			\hline		\hline
 			$X$ &
 			$\pattern{global scale = 0.8}{1}{1/1}{0/1,1/0}$ &  	
            \raisebox{0.6ex}{\begin{tikzpicture}[scale = 0.4, baseline=(current bounding box.center)]
 			\foreach \x/\y in {0/1,1/0}		
 			\fill[gray!20] (\x,\y) rectangle +(1,1);
 			\draw (0.01,0.01) grid (1+0.99,1+0.99);
 			\node  at (1.6,1.6) {\tiny $p_1$};
 			\filldraw (1,1) circle (4pt) node[above left] {};
 			\end{tikzpicture}}
 	        & {Irreducible} & - & - & Theorem~\ref{th:pattern X}
 			\\[8pt]
 			\hline
 			$Y$ & $\pattern{global scale = 0.8}{1}{1/1}{0/0,0/1,1/0}$ &
            \raisebox{0.3ex}{\begin{tikzpicture}[scale = 0.4, baseline=(current bounding box.center)]
 			\foreach \x/\y in {0/1,1/0,0/0}		
 			\fill[gray!20] (\x,\y) rectangle +(1,1);
 			\draw (0.01,0.01) grid (1+0.99,1+0.99);
 			\node  at (1.6,1.6) {\tiny $p_1$};
 			\filldraw (1,1) circle (4pt) node[above left] {};
 			\end{tikzpicture}}
            & {Any} & - & - & Theorem~\ref{th:pattern Y}
 			\\[8pt]
 			\hline
 			12 & $\pattern{global scale = 0.6}{2}{1/1,2/2}{0/0,0/1,0/2,1/0,2/0}$ & - & - & -
 			& - & Corollary~\ref{cor1}
 			\\[8pt]
 			\hline
 		    13 & $\pattern{global scale = 0.6}{2}{1/1,2/2}{0/0,0/1,0/2,1/0,2/0,2/1,2/2,1/2}$
 		    & {\begin{tikzpicture}[scale = 0.3, baseline=(current bounding box.center)]
 			\foreach \x/\y in {0/0,0/1,0/2,1/0,2/0,2/1,2/2,1/2}		
 			\fill[gray!20] (\x,\y) rectangle +(1,1);
 			\draw (0.01,0.01) grid (2+0.99,2+0.99);
 			\node  at (1.6,1.6) {\tiny $p_1$};
 			\filldraw (1,1) circle (4pt) node[above left] {};
            \filldraw (2,2) circle (4pt) node[above left] {};
 			\end{tikzpicture}}
            & {Any} & - & - & Theorem ~\ref{th:pattern 13}
 		    \\[-10pt]
		    & & & & & & Corollary~\ref{cor2} \\
 		    \hline
 			17 & $\pattern{global scale = 0.6}{2}{1/1,2/2}{0/0,0/1,0/2,1/0,2/0,2/1,1/2}$ & -  & - & - & - & Corollary~\ref{cor3}
 			\\[8pt]
 			\hline
 			19 & $\pattern{global scale = 0.6}{2}{1/1,2/2}{0/1,0/2,1/1,1/2,2/0,2/2}$ &
            \raisebox{0.3ex}{\begin{tikzpicture}[scale = 0.3, baseline=(current bounding box.center)]
 			\foreach \x/\y in {0/1,0/2,1/1,1/2,2/0,2/2}		
 			\fill[gray!20] (\x,\y) rectangle +(1,1);
 			\draw (0.01,0.01) grid (2+0.99,2+0.99);
 			\node  at (2.6,1.6) {\tiny $p_1$};
 			\filldraw (1,1) circle (4pt) node[above left] {};
            \filldraw (2,2) circle (4pt) node[above left] {};
 			\end{tikzpicture}}
            & {Any} & - & - & Theorem~\ref{th:pattern 19}
 			\\[8pt]
 			\hline
 			20 & $\pattern{global scale = 0.6}{2}{1/1,2/2}{0/0,0/1,0/2,1/1,1/2,2/1,2/0}$ &
            \raisebox{0.6ex}{\begin{tikzpicture}[scale = 0.3, baseline=(current bounding box.center)]
 			\foreach \x/\y in {0/0,0/1,0/2,1/1,1/2,2/1,2/0}		
 			\fill[gray!20] (\x,\y) rectangle +(1,1);
 			\draw (0.01,0.01) grid (2+0.99,2+0.99);
 			\node  at (2.6,2.6) {\tiny $p_1$};
            \node  at (1.6,0.5) {\tiny $p_2$};
 			\filldraw (1,1) circle (4pt) node[above left] {};
            \filldraw (2,2) circle (4pt) node[above left] {};
 			\end{tikzpicture}}
             & {Any} & {Any} & - & Theorem~\ref{th:pattern 20}
 			\\[8pt]
 			\hline
 			22 & $\pattern{global scale = 0.6}{2}{1/1,2/2}{0/0,0/1,1/1,1/2,2/0,2/2}$ &
            \raisebox{0.6ex}{\begin{tikzpicture}[scale = 0.3, baseline=(current bounding box.center)]
 			\foreach \x/\y in {0/0,0/1,1/1,1/2,2/0,2/2}		
 			\fill[gray!20] (\x,\y) rectangle +(1,1);
 			\draw (0.01,0.01) grid (2+0.99,2+0.99);
 			\node  at (0.6,2.6) {\tiny $p_1$};
            \node  at (1.6,0.5) {\tiny $p_2$};
            \node  at (2.6,1.6) {\tiny $p_3$};
 			\filldraw (1,1) circle (4pt) node[above left] {};
            \filldraw (2,2) circle (4pt) node[above left] {};
 			\end{tikzpicture}}
            & {Any} & {Any} & {Any} & Theorem~\ref{th:pattern 22}
 			\\[8pt]
 	        \hline
 			28 & $\pattern{global scale = 0.6}{2}{1/1,2/2}{0/0,0/1,1/0,1/2,2/1,2/2}$ &
            {\begin{tikzpicture}[scale = 0.3, baseline=(current bounding box.center)]
 			\foreach \x/\y in {0/0,0/1,1/0,1/2,2/1,2/2}		
 			\fill[gray!20] (\x,\y) rectangle +(1,1);
 			\draw (0.01,0.01) grid (2+0.99,2+0.99);
 			\node  at (1.6,1.6) {\tiny $p_1$};
 			\filldraw (1,1) circle (4pt) node[above left] {};
            \filldraw (2,2) circle (4pt) node[above left] {};
 			\end{tikzpicture}}
            & {Any} & - & - & Theorem~\ref{th:pattern 28}
 			\\[8pt]
 			\hline
 			33 & $\pattern{global scale = 0.6}{2}{1/1,2/2}{0/1,0/2,1/0,2/0,2/1,1/2}$ &
            \raisebox{0.6ex}{\begin{tikzpicture}[scale = 0.3, baseline=(current bounding box.center)]
 			\foreach \x/\y in {0/1,0/2,1/0,2/0,2/1,1/2}		
 			\fill[gray!20] (\x,\y) rectangle +(1,1);
 			\draw (0.01,0.01) grid (2+0.99,2+0.99);
 			\node  at (2.6,2.6) {\tiny $p_1$};
 			\filldraw (1,1) circle (4pt) node[above left] {};
            \filldraw (2,2) circle (4pt) node[above left] {};
 			\end{tikzpicture}} & {Irreducible} & - & - & Theorem~\ref{th:pattern 33}
 			\\[8pt]
 		    \hline
 			66 & $\pattern{global scale = 0.6}{2}{1/1,2/2}{0/0,0/1,0/2,1/0,2/0,1/1}$ & -  & - & - & - & Corollary~\ref{cor4}
 			\\[8pt]
           \hline
 			\cline{1-6}
 		\end{tabular}
\end{center} 	}
 	\caption{Distributions of generalizations of short mesh patterns. Note that pattern $X$ is essentially equivalent to pattern $Z$ in the sense of distribution.}
 	\label{tab-1}
\end{table}

\begin{table}[htbp]
{
		\renewcommand{\arraystretch}{1.6}
	\begin{center}
	\begin{tabular}{|c|c|c|c|c|c|}
	  	\hline {Nr.\ } & {Repr.\ $p$}  & {Generalization}  &  {\ $p_1$ } & {\ $p_2$}  & {Avoidance}
		\\[5pt]
		\hline \hline
		27 &    $\pattern{scale = 0.6}{2}{1/1,2/2}{0/1,0/2,1/0,1/1,2/0,2/2}$ &	
		    \raisebox{0.3ex}{\begin{tikzpicture}[scale = 0.3, baseline=(current bounding box.center)]
 			\foreach \x/\y in {0/1,0/2,1/0,1/1,2/0,2/2}		
 			\fill[gray!20] (\x,\y) rectangle +(1,1);
 			\draw (0.01,0.01) grid (2+0.99,2+0.99);
 			\node  at (0.6,0.6) {\tiny $p_1$};
            \node  at (2.6,1.6) {\tiny $p_2$};
 			\filldraw (1,1) circle (4pt) node[above left] {};
            \filldraw (2,2) circle (4pt) node[above left] {};
 			\end{tikzpicture}}  & {Any}  & {Any}
		&	Theorem~\ref{th:pattern 27}
		\\[5pt]
		\hline

		28 & $\pattern{scale = 0.6}{2}{1/1,2/2}{0/0,0/1,1/0,1/2,2/1,2/2}$	 & 	
		    \raisebox{0.3ex}{\begin{tikzpicture}[scale = 0.3, baseline=(current bounding box.center)]
 			\foreach \x/\y in {0/0,0/1,1/0,1/2,2/1,2/2}		
 			\fill[gray!20] (\x,\y) rectangle +(1,1);
 			\draw (0.01,0.01) grid (2+0.99,2+0.99);
 			\node  at (1.6,1.6) {\tiny $p_1$};
            \node  at (2.6,0.6) {\tiny $p_2$};
 			\filldraw (1,1) circle (4pt) node[above left] {};
            \filldraw (2,2) circle (4pt) node[above left] {};
 			\end{tikzpicture}} & {Any} & {Any}	
	    &   Theorem~\ref{th:pattern 28-2}
		\\[5pt]
		\hline
		30 &   $\pattern{scale = 0.6}{2}{1/1,2/2}{0/1,0/2,1/0,1/1,1/2,2/0,2/1}$  &			    \raisebox{0.6ex}{\begin{tikzpicture}[scale = 0.3, baseline=(current bounding box.center)]
 			\foreach \x/\y in {0/1,0/2,1/0,1/1,1/2,2/0,2/1}		
 			\fill[gray!20] (\x,\y) rectangle +(1,1);
 			\draw (0.01,0.01) grid (2+0.99,2+0.99);
 			\node  at (2.6,2.6) {\tiny $p_1$};
 			\filldraw (1,1) circle (4pt) node[above left] {};
            \filldraw (2,2) circle (4pt) node[above left] {};
 			\end{tikzpicture}}  &	
		{Two general classes} &  -  &   Theorem~\ref{th:pattern 30}
		\\[5pt]	
		\cline{1-6}
	\end{tabular}
	\end{center}
}
	\caption{Avoidance for generalizations of short mesh patterns.}
 	\label{tab-2}
\end{table}

For a more sophisticated example illustrating the power of our results in this paper, suppose that one wants to find the distribution of the mesh pattern in Figure~\ref{sophist-ex-mesh}. Approaching this problem directly is probably not doable. However, one can see that the elements $a$ and $b$ give a mesh pattern of the form in Figure~\ref{pic-thm-pat-19}, so that Theorem~\ref{th:pattern 19} can be
applied with $p_1$ there given by the elements $c$ and $d$, which give a mesh
pattern of the form in Figure~\ref{pic-thm-pat-28}, so that Theorem~\ref{th:pattern 28} can be applied. Finally, $p_1$ in Theorem~\ref{th:pattern 28} in our example is nothing else but the mesh pattern Nr.\ 66, so Corollary~\ref{cor3} can be applied.  This will result in the distribution of the mesh pattern in Figure~\ref{sophist-ex-mesh} be
\begin{small}
$$F(x)\left(1+x(F(x)-1)\left(\frac{1}{1+x^2F(x)\left(F(x)-1-qx\sum_{n=1}^{\infty}\prod_{i=1}^{n-1}(q+i)x^n \right)}-1\right)\right), $$
\end{small}
where $F(x):=\sum_{n\geq 0}n!x^n$.
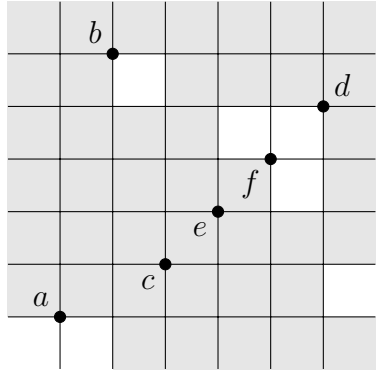
\begin{figure}[!ht]
	\begin{center}
		\begin{tikzpicture}[scale=0.7, baseline=(current bounding box.center)]
		\foreach \x/\y in {0/1,0/2,0/3,0/4,0/5,0/6,
			1/1,1/2,1/3,1/4,1/5,1/6,
			2/0,2/1,2/2,2/3,2/4,2/6,
			3/0,3/1,3/2,3/3,3/4,3/5,3/6,
			4/0,4/1,4/2,4/3,4/5,4/6,
			5/0,5/1,5/2,5/5,5/6,
			6/0,6/2,6/3,6/4,6/5,6/6}
		\fill[gray!20] (\x,\y) rectangle +(1,1);
		\draw (0.01,0.01) grid (6+0.99,6+0.99);
		\filldraw (1,1) circle (3pt) node[above left] {$a$};
		\filldraw (2,6) circle (3pt) node[above left] {$b$};
		\filldraw (3,2) circle (3pt) node[below left] {$c$};
		\filldraw (6,5) circle (3pt) node[above right] {$d$};
		\filldraw (4,3) circle (3pt) node[below left] {$e$};
		\filldraw (5,4) circle (3pt) node[below left] {$f$};

		\end{tikzpicture}
\caption{A mesh pattern of length 6}\label{sophist-ex-mesh}
	\end{center}
\end{figure}

In Tables~\ref{tab-1} and~\ref{tab-2} we give references to our enumerative results related to distribution and avoidance, respectively. Moreover, in Table~\ref{tab-3} we give references to our distribution and avoidance results on certain generalizations of short mesh patterns. We note that the patterns $p_1$, $p_2$, $p_3$ in Tables~\ref{tab-1}--\ref{tab-3} can be empty, in which case one needs to substitute the generating functions $A_{p_i}(x)$ and $F_{p_i}(x,q)$ by 0 and $qF(x)$, respectively, in our results. Indeed, one can assume that any permutation contains exactly one occurrence of the empty pattern, which makes the substitutions work. Also, in the case of empty $p_1$, one needs to set $k=1$ in our results related to  the pattern Nr.\ 34 (in Theorems~\ref{thm-pat-34} and \ref{th:pattern 34-2}). In this way, one can obtain  any previous results in \cite{Hilmarsson2015Wilf,SZ} related to the patterns appearing in this paper.

\begin{table}[htbp]
{
		\renewcommand{\arraystretch}{1.7}
	\begin{center}
	\begin{tabular}{|c|c|c|c|c|c|}
	  	\hline {Nr.\ } & {Repr.\ $p$}  & {Generalization}  &  {\ $p_1$ } & {\ $p_2$} &  {Reference}
		\\[5pt]
		\hline \hline
		 33 & $\pattern{scale = 0.6}{2}{1/1,2/2}{0/1,0/2,1/0,2/0,2/1,1/2}$ & \raisebox{0.6ex}{\begin{tikzpicture}[scale = 0.3, baseline=(current bounding box.center)]
 			\foreach \x/\y in {0/1,0/2,1/0,2/0,2/1,1/2}		
 			\fill[gray!20] (\x,\y) rectangle +(1,1);
 			\draw (0.01,0.01) grid (2+0.99,2+0.99);
 			\node  at (2.6,2.6) {\tiny $p_1$};
            \node  at (1.5,1.65) {\tiny $\iddots$};
 			\filldraw (1,1) circle (4pt) node[above left] {};
            \filldraw (2,2) circle (4pt) node[above left] {};
 			\end{tikzpicture}} & {Irreducible} & - & Theorem~\ref{th:pattern 33-2}
 			\\[5pt]
            \hline
		34 &  $\pattern{scale = 0.6}{2}{1/1,2/2}{0/0,0/1,1/0,1/1,1/2,2/1,2/2}$ &
		\raisebox{0.6ex}{\begin{tikzpicture}[scale = 0.3, baseline=(current bounding box.center)]
 			\foreach \x/\y in {0/0,0/1,1/0,1/1,1/2,2/1,2/2}		
 			\fill[gray!20] (\x,\y) rectangle +(1,1);
 			\draw (0.01,0.01) grid (2+0.99,2+0.99);
            \node  at (1.5,1.5) {\tiny $p_1$};
 			\filldraw (1,1) circle (4pt) node[above left] {};
 			\end{tikzpicture}}
		&  {Any} &	-  & 	Theorem~\ref{thm-pat-34}
        \\[5pt]
        \hline
		34 &  $\pattern{scale = 0.6}{2}{1/1,2/2}{0/0,0/1,1/0,1/1,1/2,2/1,2/2}$ & \raisebox{0.3ex}{\begin{tikzpicture}[scale = 0.3, baseline=(current bounding box.center)]
 			\foreach \x/\y in {0/0,0/1,1/0,1/1,1/2,2/1,2/2}		
 			\fill[gray!20] (\x,\y) rectangle +(1,1);
 			\draw (0.01,0.01) grid (2+0.99,2+0.99);
            \node  at (1.5,1.5) {\tiny $p_1$};
            \node at (2.5,0.5) {\tiny $p_2$};
 			\filldraw (1,1) circle (4pt) node[above left] {};
 			\end{tikzpicture}}
		&  {Any} &	{Any}  & Theorem~\ref{th:pattern 34-2}
        \\[5pt]
		\cline{1-6}
	\end{tabular}
	\end{center}
}
	\caption{Generalizations of short mesh patterns allowing replacement of $\tau=12$ either by an increasing permutation or by a permutation beginning with the smallest element with all boxes shaded.  
The distributions are given for Nr.\ 33, and 34 with $p_1$ but without $p_2$, and the avoidance for Nr.\ 34 with $p_1$ and $p_2$.}
 	\label{tab-3}
\end{table}

In this paper, we need the following result.

\begin{thm}[{\cite[Theorem 1.1]{SZ}}]\label{thm-length-1} Let
 $$F(x,q) =\sum_{n\geq 0}x^n\sum_{\pi\in S_n}q^{\pattern{scale=0.5}{1}{1/1}{0/1,1/0}(\pi)}=\sum_{n\geq 0}x^n\sum_{\pi\in S_n}q^{\pattern{scale=0.5}{1}{1/1}{0/0,1/1}(\pi)}$$ and let $A(x)$ be the generating function for $S(\pattern{scale=0.5}{1}{1/1}{0/1,1/0}\ )=S(\pattern{scale=0.5}{1}{1/1}{0/0,1/1}\ )$. Then,
 $$A(x)=\frac{F(x)}{1+xF(x)},\ \ \ \ \ F(x,q)=\frac{F(x)}{1+x(1-q)F(x)}.$$
\end{thm}

This paper is organized as follows. In Section~\ref{sec2} we present  distribution and avoidance results for certain mesh patterns derived from the pattern $Y$. In Section~\ref{sec3} we study the distribution and avoidance  for certain patterns derived from the patterns Nr.\ 13, 19, 20, 22, and 28. Section~\ref{sec4} is dedicated to the  distribution and avoidance  for certain mesh patterns derived from the pattern $X$. In Section~\ref{sec5} we deal with avoidance results for certain mesh patterns  derived from the patterns Nr.\ 27, 28, 30, and 34. Finally, we provide some concluding remarks in Section~\ref{final-sec}.

\section{A generalization of the pattern $Y$}\label{sec2}
In this section, we consider a generalization of the pattern $Y = \pattern{scale = 0.8}{1}{1/1}{0/0,0/1,1/0}$. As an application of our general results, we will find the distributions of the following patterns:\\

Nr.~12 = $\pattern{scale = 0.6}{2}{1/1,2/2}{0/0,0/1,0/2,1/0,2/0}$, \quad
Nr.~13 = $\pattern{scale = 0.6}{2}{1/1,2/2}{0/0,0/1,0/2,1/0,2/0,2/1,2/2,1/2}$, \quad
Nr.~17 = $\pattern{scale = 0.6}{2}{1/1,2/2}{0/0,0/1,0/2,1/0,2/0,2/1,1/2}$, \quad
Nr.~66 = $\pattern{scale=0.6}{2}{1/1,2/2}{0/0,0/1,0/2,1/0,2/0,1/1}$.\\

\begin{thm}\label{th:pattern Y}
Suppose that $p$ is the pattern shown in Figure~\ref{pic-thm-pat-Y}, where $p_1$ is any mesh pattern, and the label $a$ is to be ignored. Then,
\begin{align}		
	A_p(x)& =(1-x)F(x)+xA_{p_1}(x),\label{avoidance-patt-Y-2}\\[6pt]
	F_p(x,q)& =(1-x)F(x)+xF_{p_1}(x,q).\label{dis-pattern-Y-2}
\end{align}
\end{thm}

\begin{figure}[!ht]
\begin{center}
    \begin{tikzpicture}[scale=0.8, baseline=(current bounding box.center)]
	\foreach \x/\y in {0/0,0/1,1/0}		
		\fill[gray!20] (\x,\y) rectangle +(1,1);
	\draw (0.01,0.01) grid (1+0.99,1+0.99);
    \filldraw (1,1) circle (3pt) node[below left] {$a$};
	\node  at (1.5,1.5) {$p_1$};
	\end{tikzpicture}
\caption{Related to the proof of Theorem~\ref{th:pattern Y}.}\label{pic-thm-pat-Y}
\end{center}
\end{figure}
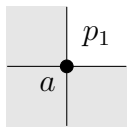

\begin{proof}
Any permutation counted by $F(x)$ either avoids $p$, which is counted by $A_p(x)$, or contains at least one occurrence of $p$. The generating function for the latter case is  $x(F(x)-A_{p_1}(x))$. Indeed, for any occurrence of $p$, the element $a$ in Figure~\ref{pic-thm-pat-Y} is the same.  Now the North East box in the figure must contain at least one occurrence of $p_1$, which is counted by $F(x)-A_{p_1}(x)$, and the element $a$ contributes the factor of $x$. This leads  to
\begin{equation}\label{av-pattern-Y}
A_p(x) +x(F(x)-A_{p_1}(x))=F(x).
\end{equation}

For the distribution, we have the following functional equation:
\begin{equation}\label{dis-pattern-Y}
A_p(x) +x(F_{p_1}(x,q)-A_{p_1}(x))=F_p(x,q).
\end{equation}
Our proof of \eqref{dis-pattern-Y} is essentially the same as that  in the avoidance case.
In particular, the contribution of the North East box is $F_{p_1}(x,q)-A_{p_1}(x)$, since every occurrence of $p_1$ there, along with the element $a$, will give an occurrence of $p$, and all occurrences of $p$ are obtained in this way.
The formulas \eqref{avoidance-patt-Y-2} and \eqref{dis-pattern-Y-2} now follow from the formulas \eqref{av-pattern-Y} and \eqref{dis-pattern-Y}, respectively. This completes the proof.
\end{proof}

The following results follow from Theorem~\ref{th:pattern Y}.

\begin{coro}[{\cite[Theorem 2.4]{SZ}}]\label{cor1} For the pattern Nr.~$12$ = $p=\pattern{scale=0.6}{2}{1/1,2/2}{0/0,0/1,0/2,1/0,2/0}$, we have
	\begin{align*}		
A_p(x) & =(1-x)F(x)+x, \\[6pt]
F_p(x,q) & =(1-x)F(x)+xF(qx).
\end{align*}
\end{coro}

\begin{proof} The mesh pattern $p$ is obtained from the pattern $Y$ by inserting  the pattern $p_1=\pattern{scale=0.8}{1}{1/1}{}$. It is easy to see that $A_{p_1}(x)=1$ and $F_{p_1}(x,q)=F(qx)$.
After substituting these into \eqref{avoidance-patt-Y-2} and \eqref{dis-pattern-Y-2}, we obtain the desired result. \end{proof}

\begin{coro}[{\cite[Theorem 2.5]{SZ}}]\label{cor2} For the pattern Nr.~$13$ = $p=\pattern{scale=0.6}{2}{1/1,2/2}{0/0,0/1,2/0,2/2,1/0,0/2,1/2,2/1}$, we have
\begin{align*}		
A_p(x)& =(1-x^2)F(x),\\[6pt]
F_p(x,q)& =(1-x^2+qx^2)F(x).
\end{align*}
\end{coro}

\begin{proof} The mesh pattern $p$ is obtained from the pattern $Y$ by inserting  the pattern $p_1=\pattern{scale=0.8}{1}{1/1}{0/1,1/0,1/1}$. It is easy to see that $A_{p_1}(x)=1$ and $F_{p_1}(x,q)=F(qx)$.
Together with  $A_{p_1}(x)=F(x)-xF(x)$, and $F_{p_1}(x,q)=F(x)-xF(x)+ qxF(x)$, we have the desired result. \end{proof}

\begin{coro}[{\cite[Theorem 3.2]{SZ}}]\label{cor3} For the pattern Nr.~$17$ = $p=\pattern{scale = 0.6}{2}{1/1,2/2}{0/1,1/2,0/0,2/0,1/0,0/2,2/1}$, we have
$$F_p(x,q)=\left(1-x + \frac{x}{1+x(1-q)F(x)}\right)F(x).$$
Also, $A_p(x)=F_p(x,0)$.
\end{coro}

\begin{proof} The mesh pattern $p$ is obtained from the pattern $Y$ by inserting  the pattern $p_1=\pattern{scale=0.8}{1}{1/1}{0/1,1/0}$. By Theorem~\ref{thm-length-1}, we have $F_{p_1}(x,q)=\frac{F(x)}{1+x(1-q)F(x)}$. After substituting it into~\eqref{dis-pattern-Y-2}, we obtain the desired result.\end{proof}

The following result is new.

\begin{coro}\label{cor4} For the pattern Nr.~$66$ = $p=\pattern{scale=0.6}{2}{1/1,2/2}{0/0,0/1,2/0,1/0,0/2,1/1}$, we have
\begin{align*}		
A_p(x) & =(1-x)F(x)+x, \\
F_p(x,q) & =(1-x)F(x)+x+x\sum_{n=1}^{\infty}\prod_{i=0}^{n-1}(q+i)x^n.
\end{align*}
\end{coro}

\begin{proof} The mesh pattern $p$ is obtained from the pattern $Y$ by inserting  the pattern $p_1=\pattern{scale=0.8}{1}{1/1}{0/0}$. Note that  $A_{p_1}(x)=1$ and the distribution of $p_1$ is given by the {\em unsigned Stirling numbers of the first kind} \cite[p. 19]{Stanley}:
\begin{align*}
F_{p_1}(x,q)=1+\sum_{n=1}^{\infty}\prod_{i=0}^{n-1}(q+i)x^n.
\end{align*} Substituting these into \eqref{avoidance-patt-Y-2} and \eqref{dis-pattern-Y-2} gives the desired result. \end{proof}

\section{Distributions of several mesh patterns}\label{sec3}
In this section, we generalize known distribution results for the patterns

\begin{center}
Nr. 13 =  $\pattern{scale = 0.6}{2}{1/1,2/2}{0/0,0/1,1/0,0/2,2/0,2/2,2/1,1/2}$, \quad
Nr. 19 = $\pattern{scale = 0.6}{2}{1/1,2/2}{0/1,0/2,1/1,1/2,2/0,2/2}$,  \quad
Nr. 20 = $\pattern{scale = 0.6}{2}{1/1,2/2}{0/0,0/1,0/2,1/1,1/2,2/1,2/0}$, \\[6pt]

Nr. 22 = $\pattern{scale = 0.6}{2}{1/1,2/2}{0/0,0/1,1/1,1/2,2/0,2/2}$,  \quad
Nr. 28 = $\pattern{scale=0.6}{2}{1/1,2/2}{0/0,0/1,1/0,1/2,2/1,2/2}$.
\end{center}
The theorems proved in this section are not so difficult (and somewhat similar), but they prepare the reader for the upcoming more involved distribution or avoidance results.

\subsection{The pattern Nr. 13}
We first consider the generalization of the pattern Nr. 13 = $\pattern{scale = 0.6}{2}{1/1,2/2}{0/0,0/1,1/0,0/2,2/0,2/2,2/1,1/2}$.
\begin{thm}\label{th:pattern 13}
Suppose that $p$ is the pattern shown in Figure~\ref{pic-thm-pat-13}, where $p_1$ is any mesh pattern, and the labels $a$ and $b$ are to be ignored. Then, \begin{align*}		
A_p(x) & =F(x)-x^2(F(x)-A_{p_1}(x)),\\[6pt]
	F_p(x,q) & = (1-x^2)F(x)+x^2 F_{p_1}(x,q).
	\end{align*}
\end{thm}

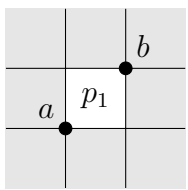
\begin{figure}[!ht]
\begin{center}
	\begin{tikzpicture}[scale=0.8, baseline=(current bounding box.center)]
	\foreach \x/\y in {0/0,0/1,1/0,0/2,2/0,2/2,2/1,1/2}		
	\fill[gray!20] (\x,\y) rectangle +(1,1);
	\draw (0.01,0.01) grid (2+0.99,2+0.99);
	\filldraw (1,1) circle (3pt) node[above left] {$a$};
	\filldraw (2,2) circle (3pt) node[above right] {$b$};
	\node at (1.5,1.5) {$p_1$};
	\end{tikzpicture}
\caption{Related to the proof of Theorem~\ref{th:pattern 13}.}\label{pic-thm-pat-13}
\end{center}
\end{figure}

\begin{proof}
We have the following functional equation:
\begin{equation}\label{av-pattern-13}
A_p(x) +x^2(F(x)-A_{p_1}(x))=F(x).
\end{equation}
Indeed, any permutation counted by $F(x)$ either avoids $p$, which is counted by $A_p(x)$, or contains at least one occurrence of $p$. The generating function for the latter case is  $x^2(F(x)-A_{p_1}(x))$. Indeed, an occurrence of $p$ implies at least one occurrence of $p_1$ in the central box in Figure~\ref{pic-thm-pat-13}, which contributes the factor of $F(x)-A_{p_1}(x)$. Besides, the factor of $x^2$ is contributed by $ab$.
	
For the distribution, we have the following functional equation:
\begin{equation}\label{dis-pattern-13}
A_p(x) +x^2(F_{p_1}(x,q)-A_{p_1}(x))=F_p(x,q).
\end{equation}
The proof of \eqref{dis-pattern-13} is essentially the same as that  in the avoidance case. In particular, the factor of $F_{p_1}(x,q)-A_{p_1}(x)$ comes from the fact that in any occurrence of $p$, $a$, and $b$ in Figure~\ref{pic-thm-pat-13} must be the same.
The desired result follows from \eqref{av-pattern-13} and \eqref{dis-pattern-13}.
This completes the proof.
\end{proof}

\subsection{The pattern Nr. 19}

Now we generalize the pattern Nr. 19 = $\pattern{scale = 0.6}{2}{1/1,2/2}{0/1,0/2,1/1,1/2,2/0,2/2}$.

\begin{thm}\label{th:pattern 19}
Suppose that $p$ is the pattern shown in Figure~\ref{pic-thm-pat-19}, where $p_1$ is any mesh pattern, and the labels $a$, $b$, $A$, and $B$ are to be ignored. Then, \begin{align*}		
A_p(x)&=F(x)-x(F(x)-1)(F(x)-A_{p_1}(x)),\\[6pt]
F_p(x,q) &= F(x)+x(F(x)-1)(F_{p_1}(x,q)-F(x)).
	\end{align*}
\end{thm}

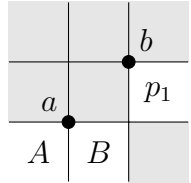
\begin{figure}[!ht]
\begin{center}
	\begin{tikzpicture}[scale=0.8, baseline=(current bounding box.center)]
	\foreach \x/\y in {0/1,0/2,1/1,1/2,2/0,2/2}		
	\fill[gray!20] (\x,\y) rectangle +(1,1);
	\draw (0.01,0.01) grid (2+0.99,2+0.99);
	\filldraw (1,1) circle (3pt) node[above left] {$a$};
	\filldraw (2,2) circle (3pt) node[above right] {$b$};
	\node at (0.5,0.5) {$A$};
    \node at (1.5,0.5) {$B$};
    \node at (2.5,1.5) {$p_1$};
	\end{tikzpicture}
\caption{Related to the proof of Theorem~\ref{th:pattern 19}.}\label{pic-thm-pat-19}
\end{center}
\end{figure}

\begin{proof}
We have the following functional equation:
\begin{equation}\label{av-pattern-19}
A_p(x)+x(F(x)-1)(F(x)-A_{p_1}(x))=F(x),
\end{equation}
where the right hand side counts all permutations. Indeed, each permutation either avoids $p$, which is given by the $A_p(x)$ term in \eqref{av-pattern-19}, or it contains at least one occurrence of $p$. The latter case is given by the second term on the left hand side of \eqref{av-pattern-19}, which shall be proved in the following.
To pick the occurrence $ab$ in Figure~\ref{pic-thm-pat-19}, we choose $b$  the \emph{highest} possible and $a$ is then uniquely determined. Referring to this figure, we note that the East box must contain at least one occurrence of $p_1$, which is counted by $F(x)-A_{p_1}(x)$. Furthermore, the boxes $A$ and $B$ together with $a$, which is the maximum element in the permutation $AaB$, contribute the factor of $F(x)-1$. Note that $a$ must exist, so $AaB$ is not empty and there are no other restrictions for $AaB$.
Finally, $b$ contributes the factor of $x$. Thus, we complete the proof of   \eqref{av-pattern-19} and  hence give the formula of  $A_p(x)$.
	
    For the distribution, we have the following functional equation:
\begin{equation}\label{dis-pattern-19}
A_p(x)+x(F(x)-1)(F_{p_1}(x,q)-A_{p_1}(x))=F_p(x,q).
\end{equation}
The proof of \eqref{dis-pattern-19} is essentially the same as that in the avoidance case.
Since $a$ is uniquely determined, the boxes $A$ and $B$ together with $a$ contribute the factor of $F(x)-1$. The East box contributes the factor of $F_{p_1}(x,q)-A_{p_1}(x)$, since each occurrence of $p_1$ in that box induces one  occurrence of $p$. Finally, the factor of $x$ corresponds to the element $b$, which completes the proof of \eqref{dis-pattern-19}.
Substituting the formula of $A_p(x)$  into \eqref{dis-pattern-19}   gives the formula of $F_p(x,q)$ as desired.
\end{proof}

\subsection{The pattern Nr. 20}
Now, we generalize the pattern $\pattern{scale = 0.6}{2}{1/1,2/2}{0/0,0/1,0/2,1/1,1/2,2/1,2/0}$ by adding $p_1$ and $p_2$ in it.

\begin{thm}\label{th:pattern 20}
Suppose that $p$ is the pattern shown in Figure~\ref{pic-thm-pat-20}, where $p_1$ and $p_2$ are any mesh patterns, and the labels $a$ and $b$ are to be ignored. Then,
\begin{align*}		
A_p(x)=&\, F(x)-x^2(F(x)-A_{p_1}(x))(F(x)-A_{p_2}(x)),\\[6pt]
F_p(x,q)=&\, F(x)+x^2 \big( (F_{p_1}(x,q)-A_{p_1}(x))(F_{p_2}(x,q)-A_{p_2}(x))-\\
&(F(x)-A_{p_1}(x))(F(x)-A_{p_2}(x)) \big).
\end{align*}
\end{thm}

\begin{figure}[!ht]
\begin{center}
	\begin{tikzpicture}[scale=0.8, baseline=(current bounding box.center)]
	\foreach \x/\y in {0/0,0/1,0/2,1/1,1/2,2/1,2/0}		
	\fill[gray!20] (\x,\y) rectangle +(1,1);
	\draw (0.01,0.01) grid (2+0.99,2+0.99);
	\filldraw (1,1) circle (3pt) node[above left] {$a$};
	\filldraw (2,2) circle (3pt) node[above left] {$b$};
	\node at (2.5,2.5) {$p_1$};
	\node at (1.5,0.5) {$p_2$};
	\end{tikzpicture}
\caption{Related to the proof of Theorem~\ref{th:pattern 20}.}\label{pic-thm-pat-20}
\end{center}
\end{figure}
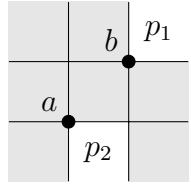

\begin{proof}
We have the following functional equation:
\begin{equation}\label{av-pattern-20}
A_p(x)+x^2(F(x)-A_{p_1}(x))(F(x)-A_{p_2}(x))=F(x).
\end{equation}
Indeed, the right hand side of \eqref{av-pattern-20} counts all permutations. On the left hand side, $A_p(x)$ gives avoidance of $p$, and the other term, to be justified next, counts permutations with at least one occurrence of $p$. Indeed, the elements $a$ and $b$ contributing the factor of $x^2$ are uniquely determined, which is evident from Figure~\ref{pic-thm-pat-20}. Referring to this figure, we note that the upper (resp., lower) non-shaded box $A$ (resp., $B$) must contain at least one occurrence of $p_1$ (resp., $p_2$) counted by $F(x)-A_{p_1}(x)$ (resp., $F(x)-A_{p_2}(x)$). Thus, we complete the proof of   \eqref{av-pattern-20} and  hence give the formula of $A_p(x)$.
	
    For the distribution, we have the following functional equation:
\begin{equation}\label{dis-pattern-20}
A_p(x)+x^2(F_{p_1}(x,q)-A_{p_1}(x))(F_{p_2}(x,q)-A_{p_2}(x))=F_p(x,q).
\end{equation}
The proof of~\eqref{dis-pattern-20} is essentially the same as that in the avoidance case.
The box $A$ (resp., $B$) contributes the factor of $F_{p_1}(x,q)-A_{p_1}(x)$ (resp., $F_{p_2}(x,q)-A_{p_2}(x)$), since each occurrence of $p_1$ in $A$, together with any occurrence of $p_2$ in $B$, form an occurrence of $p$. Together with the factor of $x^2$ corresponding to $ab$, we complete the proof of \eqref{dis-pattern-20}.
Substituting the formula of $A_p(x)$ into  \eqref{dis-pattern-20}  gives the desired result.
\end{proof}

\subsection{The pattern Nr. 22}
Next we generalize the pattern $\pattern{scale=0.6}{2}{1/1,2/2}{0/0,0/1,1/1,1/2,2/0,2/2}$ by adding $p_1$, $p_2$, and $p_3$ in it.

\begin{thm}\label{th:pattern 22}
Suppose that $p$ is the pattern shown in Figure~\ref{pic-thm-pat-22}, where $p_1$, $p_2$, and $p_3$ are  any mesh patterns, and the labels $a$ and $b$ are to be ignored. Then,
\begin{align*}		
A_p(x)  = &  \, F(x)-x^2 (F(x)-A_{p_1}(x))(F(x)-A_{p_2}(x))(F(x)-A_{p_3}(x)),\\[6pt]
F_p(x,q)  = &  \, F(x)+x^2\Big( (F_{p_1}(x,q)-A_{p_1}(x))(F_{p_2}(x,q)-A_{p_2}(x))(F_{p_3}(x,q)-A_{p_3}(x))\\
& \quad  -(F(x)-A_{p_1}(x))(F(x)-A_{p_2}(x)) (F(x)-A_{p_3}(x))\Big).
\end{align*}
\end{thm}

\begin{figure}[!ht]
\begin{center}
	\begin{tikzpicture}[scale=0.8, baseline=(current bounding box.center)]
	\foreach \x/\y in {0/0,0/1,1/1,1/2,2/0,2/2}		
	\fill[gray!20] (\x,\y) rectangle +(1,1);
	\draw (0.01,0.01) grid (2+0.99,2+0.99);
	\filldraw (1,1) circle (3pt) node[above left] {$a$};
	\filldraw (2,2) circle (3pt) node[above right] {$b$};
	\node at (0.5,2.5) {$p_1$};
	\node at (1.5,0.5) {$p_2$};
	\node at (2.5,1.5) {$p_3$};
	\end{tikzpicture}
\caption{Related to the proof of Theorem~\ref{th:pattern 22}.}\label{pic-thm-pat-22}
\end{center}
\end{figure}
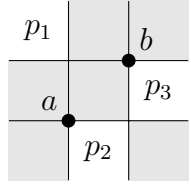

\begin{proof}
Any permutation, counted by $F(x)$, either avoids $p$, which is counted by $A_p(x)$, or contains at least one occurrence of $p$. In the latter situation, the elements $ab$  in an occurrence of $p$ are uniquely defined. Indeed, looking at Figure~\ref{pic-thm-pat-22} we see that another occurrence of $p$ cannot be inside of any of the unshaded boxes:
\begin{itemize}
\item if $p$ would occur in the box labeled with $p_1$, it would contradict the South East box being shaded;
\item if $p$ would occur in the box labeled with $p_2$, it would contradict the element $b$ being in the North East shaded box;
\item  if $p$ would occur in the box labeled with $p_3$,  it would contradict the element $a$ being in the South West shaded box.
\end{itemize}
Moreover, another occurrence of $p$ cannot begin at the box 
\begin{itemize}
\item labeled $p_1$ because of the two shaded top boxes;
\item labeled $p_2$ because of the element $a$;
\item labeled $p_3$ because of the North East box being shaded.
\end{itemize}
 So, $a$ and $b$  will contribute the factor of $x^2$. Referring to this figure, we note that the non-shaded box labeled by $p_1$ (resp., $p_2$ and $p_3$) must contain at least one occurrence of  $p_1$ (resp., $p_2$ and $p_3$) and thus is counted by $F(x)-A_{p_1}(x)$ (resp., $F(x)-A_{p_2}(x)$ and $F(x)-A_{p_3}(x)$). Thus, we obtain
\begin{equation*}
A_p(x)+x^2(F(x)-A_{p_1}(x))(F(x)-A_{p_2}(x))(F(x)-A_{p_3}(x))=F(x),
\end{equation*}
which leads to  the formula of $A_p(x)$.

    For the distribution, we have the following functional equation:
\begin{align}
 A_p(x) +  x^2  \big( F_{p_1}(x,q)  -  A_{p_1}(x) \big)  \big( F_{p_2}(x,q)-A_{p_2}(x)\big)   & \big(F_{p_3}(x,q)-A_{p_3}(x)\big)   \notag \\[5pt]
   = &\,  F_p(x,q). \label{dis-pattern-22}
\end{align}
The proof of \eqref{dis-pattern-22} is essentially the same as that in the avoidance case. The only remark is that any occurrence of the pair $(p_1, p_2, p_3)$ in their respective non-shaded boxes induce an occurrence of $p$.
Substituting the formula of $A_p(x)$ into \eqref{dis-pattern-22}   gives the formula of $F_p(x,q)$.
This completes the proof.
\end{proof}

\subsection{The pattern Nr. 28}
We now generalize the pattern $\pattern{scale=0.6}{2}{1/1,2/2}{0/0,0/1,1/0,1/2,2/1,2/2}$ by adding $p_1$ in it.

\begin{thm}\label{th:pattern 28}
Suppose that $p$ is the pattern shown in Figure~\ref{pic-thm-pat-28}, where $p_1$ is any mesh pattern, and the labels $a$, $b$, $A$, and $B$ are to be ignored. Then,
\begin{align*}		
A_p(x)& =\frac{F(x)}{1+x^2 \big(F(x)-A_{p_1}(x) \big)F(x)},\\[6pt]
F_p(x,q)& =\frac{F(x)}{1+x^2\big( F(x)-F_{p_1}(x,q) \big)  F(x)  }.
\end{align*}
\end{thm}

\begin{figure}[!ht]
\begin{center}
	\begin{tikzpicture}[scale=0.8, baseline=(current bounding box.center)]
	\foreach \x/\y in {0/0,0/1,1/0,1/2,2/1,2/2}		
	\fill[gray!20] (\x,\y) rectangle +(1,1);
	\draw (0.01,0.01) grid (2+0.99,2+0.99);
	\filldraw (1,1) circle (3pt) node[above left] {$a$};
	\filldraw (2,2) circle (3pt) node[above right] {$b$};
	\node at (0.5,2.5) {$A$};
    \node at (2.5,0.5) {$B$};
    \node at (1.5,1.5) {$p_1$};
	\end{tikzpicture}
\caption{Related to the proof of Theorem~\ref{th:pattern 28}.}\label{pic-thm-pat-28}
\end{center}
\end{figure}
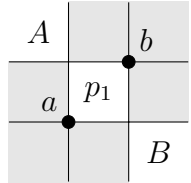

\begin{proof}
Similar to the avoidance case of Theorem~\ref{th:pattern 19}, we have
\begin{equation}\label{av-pattern-28}
A_p(x)+x^2 A_p(x) \big(F(x)-A_{p_1}(x) \big) F(x)=F(x).
\end{equation}
The only term requiring an explanation here is $x^2 A_p(x) \big(F(x)-A_{p_1}(x) \big) F(x)$ corresponding to permutations with at least one occurrence of $p$. Among all such occurrences consider $ab$ with  \emph{leftmost} possible $a$ as shown in Figure~\ref{pic-thm-pat-28}. Note that the element $b$ is then uniquely determined. Further, the middle box must contain at least one occurrence of $p_1$, which is counted by $F(x)-A_{p_1}(x)$, and the permutation in box $A$ must be $p$-avoiding since $a$ is the leftmost. Moreover, box $B$ can contain any permutation, and thus its contribution is  $F(x)$.
Finally, $a$ and $b$ contribute the factor of $x^2$, which completes the proof of   \eqref{av-pattern-28} and  hence gives the formula of $A_p(x)$.
	
For the distribution, we have the following functional equation:
\begin{equation}\label{dis-pattern-28}
A_p(x)+x^2 A_p(x) \big( F_{p_1}(x,q)-A_{p_1}(x) \big) F_p(x,q)=F_p(x,q).
\end{equation}
The proof of \eqref{dis-pattern-28} is essentially the same as that in the avoidance case. We just remark that any occurrence of $p_1$ in the middle box gives an occurrence of $p$, which explains the term $F_{p_1}(x,q)-A_{p_1}(x)$, and $F_p(x,q)$ on the left hand side of \eqref{dis-pattern-28} records occurrences of $p$ in box $B$.
\end{proof}

\section{Patterns derived from the pattern $X$}\label{sec4}
This section gives the distribution of an infinite family of mesh patterns obtained from the pattern $X=\pattern{scale = 0.8}{1}{1/1}{0/1,1/0}$. Namely, we generalize $X$ by considering \raisebox{0.6ex}{\begin{tikzpicture}[scale = 0.3, baseline=(current bounding box.center)]
 			\foreach \x/\y in {0/1,1/0}		
 			\fill[gray!20] (\x,\y) rectangle +(1,1);
 			\draw (0.01,0.01) grid (1+0.99,1+0.99);
 			\node  at (1.6,1.6) {\tiny $p_1$};
 			\filldraw (1,1) circle (4pt) node[above left] {};
 			\end{tikzpicture}},
where $p_1$ is an {\em irreducible} mesh pattern. Moreover, we use the same approach to study the distribution of a family of patterns that can be derived from the pattern  Nr.\ 33 = $\pattern{scale = 0.6}{2}{1/1,2/2}{0/1,0/2,1/0,2/0,1/2,2/1}$. Such distributions cannot be directly obtained from our results for \raisebox{0.6ex}{\begin{tikzpicture}[scale = 0.3, baseline=(current bounding box.center)]
 			\foreach \x/\y in {0/1,1/0}		
 			\fill[gray!20] (\x,\y) rectangle +(1,1);
 			\draw (0.01,0.01) grid (1+0.99,1+0.99);
 			\node  at (1.6,1.6) {\tiny $p_1$};
 			\filldraw (1,1) circle (4pt) node[above left] {};
 			\end{tikzpicture}},
since $p_1$ must be irreducible there.

\subsection{The pattern $X$}
In our next theorem we consider a generalization of $X$.

\begin{thm}\label{th:pattern X}
Suppose that $p$ is the pattern shown in Figure~\ref{pic-thm-pat-X}, where $p_1$ is an irreducible mesh pattern of length $\geq 2$, and the label $a$ is to be ignored. Then,
\begin{align*}		
A_p(x) & =\frac{ \big(1+xA_{p_1}(x) \big) F(x)}{1+xF(x)},\\
F_p(x,q)& =\left( 1+\sum_{i\ge 1}  x^i \prod_{j=1}^{i}\frac{F_{p_1}(x,q^j)}{1+xF_{p_1}(x,q^j)} \right) \frac{F(x)}{1+xF(x)}.
\end{align*}
\end{thm}

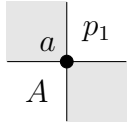
\begin{figure}[!ht]
\begin{center}
	\begin{tikzpicture}[scale=0.8, baseline=(current bounding box.center)]
	\foreach \x/\y in {0/1,1/0}		
		\fill[gray!20] (\x,\y) rectangle +(1,1);
	\draw (0.01,0.01) grid (1+0.99,1+0.99);
	\node  at (1.5,1.5) {$p_1$};
    \node at (0.5,0.5) {$A$};
    \filldraw (1,1) circle (3pt) node[above left] {$a$};
	\end{tikzpicture}
\caption{Related to the proof of Theorem~\ref{th:pattern X}.}\label{pic-thm-pat-X}
\end{center}
\end{figure}

\begin{proof}
Let $B(x)$ be the generating function for $X$-avoiding permutations. Then, it follows from Theorem~\ref{thm-length-1} that
\begin{equation}\label{th:pattern X-B(x)}
B(x)=\frac{F(x)}{1+xF(x)}.
\end{equation}
We next justify the following functional equation: \begin{equation}\label{av-pattern-X}
A_p(x)+xB(x) \big(F(x)-A_{p_1}(x)\big)=F(x).
\end{equation}
The right hand side counts all permutations. On the left hand side, $A_p(x)$ counts $p$-avoiding permutations. If a permutation contains at least one occurrence of $p$, we can consider the occurrence with the {\em leftmost} possible $a$ as shown in Figure~\ref{pic-thm-pat-X}. Referring to this figure, we note that the Noth East box must contain at least one occurrence of $p_1$, which is counted by $F(x)-A_{p_1}(x)$, and a permutation in box $A$ must be
$\pattern{scale = 0.8}{1}{1/1}{0/1,1/0}$\ -avoiding since $a$ is the leftmost. Finally, $a$ contributes the factor of $x$. This completes the proof of   \eqref{av-pattern-X}. Substituting \eqref{th:pattern X-B(x)} into \eqref{av-pattern-X},  we obtain the formula of $A_p(x)$.
	
In order to study  the distribution of $p$, we also need the distribution of $p_1$. Let $B_{p_1}(x,q)$ be the distribution of $p_1$ on $X$-avoiding permutations. Then,
   \begin{equation}\label{dis-pattern-X-p_1}
B_{p_1}(x,q)+xB_{p_1}(x,q)F_{p_1}(x,q)=F_{p_1}(x,q).
\end{equation}
This equation is obtained by considering $X$-avoiding permutations separately from the other permutations, and only the second term on the left hand side, corresponding to permutations with at least one occurrence of $X$, requires a justification.
Consider the occurrence of $X$ with the leftmost $a$ as in Figure~\ref{pic-thm-pat-X}. Any permutation in box $A$ is $X$-avoiding and thus contributes the factor of $B_{p_1}(x,q)$. Since $p_1$ is irreducible, the two non-shaded boxes in Figure~\ref{pic-thm-pat-X} are independent from each other. In other words, an occurrence of $p_1$ cannot start in $A$ and end in the other non-shaded box. Thus,  the contribution of the North East box is the factor of $F_{p_1}(x,q)$ in \eqref{dis-pattern-X-p_1}. This, along with the factor of $x$ corresponding to $a$, completes the proof of \eqref{dis-pattern-X-p_1}. Therefore, it follows from  \eqref{dis-pattern-X-p_1} that
\begin{equation}\label{dis-pattern-X-p_1-2}
B_{p_1}(x,q)=\frac{F_{p_1}(x,q)}{1+xF_{p_1}(x,q)}.
\end{equation}

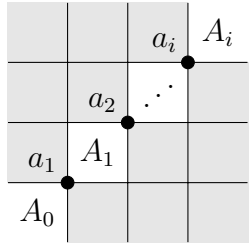
\begin{figure}[!ht]
\begin{center}
	\begin{tikzpicture}[scale=0.8, baseline=(current bounding box.center)]
	\foreach \x/\y in {0/1,0/2,0/3,1/0,1/2,1/3,2/0,2/1,2/3,3/0,3/1,3/2}		 	\fill[gray!20] (\x,\y) rectangle +(1,1);
	\draw (0.01,0.01) grid (3+0.99,3+0.99);
	 \filldraw (1,1) circle (3pt) node[above left]{$a_1$};
	 \filldraw (2,2) circle (3pt) node[above left]{$a_2$};
 	 \filldraw (3,3) circle (3pt) node[above left]{$a_i$};
 	 \node  at (0.5,0.5) {$A_0$};
 	 \node  at (1.5,1.5) {$A_1$};
     \node  at (2.5,2.6) {$\iddots$};
 	 \node  at (3.5,3.5) {$A_i$};
	\end{tikzpicture}
	\caption{Related to the proof of Theorem~\ref{th:pattern X}.}\label{pic-thm-pat-X-2}
\end{center}
\end{figure}

Now we consider the distribution of $p$. Claim that
\begin{equation}\label{dis-pattern-X}
B(x)+B(x)\sum_{i\ge 1} x^i \prod_{j=1}^{i}B_{p_1}(x,q^j)=F_{p}(x,q).
\end{equation}
Indeed, each permutation, counted by $F_{p}(x,q)$ on the right hand side, either avoids $X$, which is counted by $B(x)$ on the left hand side, or contains at least one occurrence of $X$. In the latter case, suppose $a_1$, $a_2,\ldots,a_i$ are all occurrences of $X$ in a permutation as shown in Figure~\ref{pic-thm-pat-X-2}, so all the boxes $A_j$,  where $0\leq j\leq i$, are $X$-avoiding.
Our key observation is that any occurrence of $p_1$ in a box $A_j$, where $1\leq j\leq i$, together with each of the $a_k$, where $1\leq k\leq j$, contribute an occurrence of $p$. Thus, the contribution of $A_j$ is $F_{p_1}(x,q^j)$ for $1\leq j\leq i$, and the contribution of $A_0$ is $B(x)$. Finally, $x^i$ is given by $a_i$'s and we can sum over all $i\geq 1$.  This completes the proof of \eqref{dis-pattern-X}.
Substituting the formulas of $B(x)$ and $B_{p_1}(x,q)$  into \eqref{dis-pattern-X}   gives the formula of $F_p(x,q)$. This completes the proof.
\end{proof}

Note that $p_1$ in Theorem~\ref{th:pattern X} must be of length $\geq 2$ because the result does not work for $p_1=$$\pattern{scale=0.8}{1}{1/1}{}$. In the latter case though we deal with the pattern Nr.\ 16 $= \pattern{scale = 0.6}{2}{1/1,2/2}{0/1,0/2,1/0,2/0}$, whose avoidance and distribution are solved in \cite{SZ}.

\subsection{The pattern Nr. 33}
In our next theorem we obtain a generalization of the pattern  Nr. 33 = $\pattern{scale=0.6}{2}{1/1,2/2}{0/1,0/2,1/0,2/0,1/2,2/1}$.

\begin{thm}\label{th:pattern 33}
Suppose that $p$ is the pattern shown in Figure~\ref{pic-thm-pat-33}, where $p_1$ is an irreducible mesh pattern of length $\geq 2$, and the labels $a$, $b$, $A$, and $B$ are to be ignored.
Then,

\begin{align*}		
A_p(x)&=\frac{\big( 1+x^2A_{p_1}(x) \big)F(x)}{1+x^2F(x)},\\[6pt]
F_p(x,q)&=\frac{F(x)}{1+xF(x)}+\left(\frac{F(x)}{1+xF(x)}\right)^2 \left(x+\sum_{i\ge 2} x^i  \prod_{j=2}^{i}\frac{F_{p_1}(x,q^{\binom{j}{2}})}{1+xF_{p_1}(x,q^{\binom{j}{2}})} \right).
\end{align*}
\end{thm}

\begin{figure}[!ht]
\begin{center}
	\begin{tikzpicture}[scale=0.8, baseline=(current bounding box.center)]
	\foreach \x/\y in {0/1,0/2,1/0,2/0,1/2,2/1}		
	\fill[gray!20] (\x,\y) rectangle +(1,1);
	\draw (0.01,0.01) grid (2+0.99,2+0.99);
	\filldraw (1,1) circle (3pt) node[above left] {$a$};
	\filldraw (2,2) circle (3pt) node[above left] {$b$};
	\node at (2.5,2.5) {$p_1$};
    \node at (0.5,0.5) {$A$};
    \node at (1.5,1.5) {$B$};
	\end{tikzpicture}
\caption{Related to the proof of Theorem~\ref{th:pattern 33}}\label{pic-thm-pat-33}
\end{center}
\end{figure}
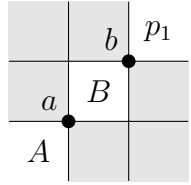


\begin{proof}
The case of avoidance is similar to our considerations of the pattern $X=\pattern{scale = 0.8}{1}{1/1}{0/1,1/0}$. We assume that the elements $a$ and $b$ in an occurrence of $p$ in Figure~\ref{pic-thm-pat-33} are {\em leftmost} possible. But then, the boxes $A$ and $B$ will be $X$-avoiding. Thus, we have the following functional equation:
\begin{equation}\label{av-pattern-33}
A_p(x)+x^2B^2(x) \big(F(x)-A_{p_1}(x) \big)=F(x).
\end{equation}
Substituting  \eqref{th:pattern X-B(x)} for $B(x)$  into \eqref{av-pattern-33}, we obtain the formula of $A_p(x)$.

For the distribution, we have the following functional equation:
\begin{equation}\label{dis-pattern-33}
B(x)+x B^2(x) + B^2(x)\sum_{i\ge 2} x^i \prod_{j=2}^{i}B_{p_1}(x,q^{\binom{j}{2}})=F_{p}(x,q),
\end{equation}
where $B_{p_1}(x,q)$ is given in \eqref{dis-pattern-X-p_1-2}. Our proof of \eqref{dis-pattern-33} is essentially the same as that of \eqref{dis-pattern-X} and we omit it. The only difference is that here we have two $X$-avoiding boxes $A$ and $B$ resulting in the factor of $B^2(x)$ instead of $B(x)$.
Substituting  \eqref{dis-pattern-X-p_1-2}  for  $B_{p_1}(x,q)$ into \eqref{dis-pattern-33}, we obtain the desired formula of $F_p(x,q)$, and thus complete the proof.
\end{proof}

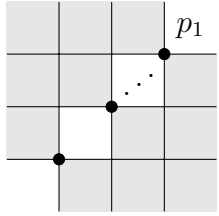
\begin{figure}[!ht]
\begin{center}
	\begin{tikzpicture}[scale=0.7, baseline=(current bounding box.center)]
	\foreach \x/\y in {0/1,0/2,1/0,1/2,2/0,2/1,0/3,1/3,2/3,3/0,3/1,3/2}
	\fill[gray!20] (\x,\y) rectangle +(1,1);
	\draw (0.01,0.01) grid (3+0.99,3+0.99);
	\filldraw (1,1) circle (3pt) node[above right] {};
    \filldraw (2,2) circle (3pt) node[above right] {};
    \filldraw (3,3) circle (3pt) node[above right] {};
    \node at (3.5,3.5) {$p_1$};
    \node at (2.5,2.6) {$\iddots$};
	\end{tikzpicture}
	\caption{Related to Theorem~\ref{th:pattern 33-2}}\label{pic-thm-pat-33-2}
\end{center}
\end{figure}

The proof of the following theorem follows  similar steps to those in Theorem~\ref{th:pattern 33}, and thus is omitted. We note that Theorem~\ref{th:pattern 33-2} is a far-reaching generalization of Theorem~\ref{th:pattern 33}.

\begin{thm}\label{th:pattern 33-2}
Suppose that $p$ is the pattern shown in Figure~\ref{pic-thm-pat-33-2}, where $p_1$ is an irreducible mesh pattern. Let pattern $X=\pattern{scale = 0.8}{1}{1/1}{0/1,1/0}$.
Then the distribution of $p$ is
\begin{align*}		
F_p(x,q)=\sum_{i=1}^{k}x^{i-1}\left(\frac{F(x)}{1+xF(x)}\right)^i+\left(\frac{F(x)}{1+xF(x)}\right)^k \sum_{i\ge k} x^i \prod_{j=k}^{i}\frac{F_{p_1}(x,q^{\binom{j}{k}})}{1+xF_{p_1}(x,q^{\binom{j}{k}})}.
\end{align*}
\end{thm}

\section{Remaining avoidance cases}\label{sec5}

In this section, we study avoidance of generalizations of the patterns:\\

Nr. 27 = $\pattern{scale=0.6}{2}{1/1,2/2}{0/1,0/2,1/0,1/1,2/0,2/2}$, \quad
Nr. 28 = $\pattern{scale=0.6}{2}{1/1,2/2}{0/0,0/1,1/0,1/2,2/1,2/2}$, \quad
Nr. 30 = $\pattern{scale=0.6}{2}{1/1,2/2}{0/1,0/2,1/0,1/1,1/2,2/0,2/1}$, \quad
Nr. 34 = $\pattern{scale=0.6}{2}{1/1,2/2}{0/0,0/1,1/0,1/1,1/2,2/1,2/2}$.

\subsection{Pattern Nr. 28}
The following theorem allows us to generalize the known avoidance result~\cite{SZ} for the pattern Nr. 28 = $\pattern{scale=0.6}{2}{1/1,2/2}{0/0,0/1,1/0,1/2,2/1,2/2}$ by inserting two mesh patterns, $p_1$ and $p_2$, into it. We were unable to find the distribution of the pattern in the next theorem because of difficulties of controlling occurrences of the patterns $p$ and $p_2$, at the same time, in the rightmost bottom box.

\begin{thm}\label{th:pattern 28-2}
Let $p$ be the pattern shown in Figure~\ref{pic-thm-pat-28-2}, where the labels $a$, $b$, and $A$ are to be ignored, and $p_1$ and $p_2$ are  any mesh patterns. Then, the avoidance of $p$ is given by
\begin{align*}		
A_p(x)=\frac{F(x)+x^2F(x)A_{p_2}(x) \big(F(x)-A_{p_1}(x) \big)}{1+x^2 \big(F(x)-A_{p_1}(x) \big) F(x)}.
\end{align*}
\end{thm}

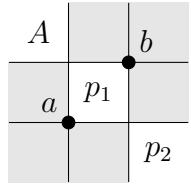
\begin{figure}[!ht]
\begin{center}
	\begin{tikzpicture}[scale=0.8, baseline=(current bounding box.center)]
	\foreach \x/\y in {0/0,0/1,1/0,1/2,2/1,2/2}		
	\fill[gray!20] (\x,\y) rectangle +(1,1);
	\draw (0.01,0.01) grid (2+0.99,2+0.99);
	\filldraw (1,1) circle (3pt) node[above left] {$a$};
	\filldraw (2,2) circle (3pt) node[above right] {$b$};
	\node at (0.5,2.5) {$A$};
    \node at (1.5,1.5) {$p_1$};
	\node at (2.5,0.5) {$p_2$};
	\end{tikzpicture}
\caption{Related to the proof of Theorem~\ref{th:pattern 28-2}}\label{pic-thm-pat-28-2}
\end{center}
\end{figure}

\begin{proof}
Let $A(x)$ be the generating function for the number of permutations avoiding  the  pattern
\raisebox{0.6ex}{\begin{tikzpicture}[scale = 0.3, baseline=(current bounding box.center)]
 			\foreach \x/\y in {0/0,0/1,1/0,1/2,2/1,2/2}		
 			\fill[gray!20] (\x,\y) rectangle +(1,1);
 			\draw (0.01,0.01) grid (2+0.99,2+0.99);
 			\node  at (1.6,1.6) {\tiny $p_1$};
 			\filldraw (1,1) circle (4pt) node[above left] {};
            \filldraw (2,2) circle (4pt) node[above left] {};
 			\end{tikzpicture}}.
Then, it follows from  Theorem~\ref{th:pattern 28} that
$$A(x)=\frac{F(x)}{1+x^2(F(x)-A_{p_1}(x))F(x)}.$$
We have the following functional equation:
\begin{equation}\label{av-pattern-28-2}
A_p(x)+x^2 A(x)\big(F(x)-A_{p_1}(x) \big) \big(F(x)-A_{p_2}(x) \big)=F(x).
\end{equation}
Indeed, each permutation, counted by $F(x)$ on the right hand side, either avoids $p$, counted by $A_p(x)$, or contains at least one occurrence of $p$. In the latter case, among all such occurrences, we can pick the occurrence $ab$ with the \emph{leftmost} possible $a$ as shown in Figure~\ref{pic-thm-pat-28-2}. Referring to this figure, we note that the central box must contain at least one occurrence of $p_1$,which is counted by $F(x)-A_{p_1}(x)$, and the rightmost bottom box must contain at least one occurrence of $p_2$, counted by $F(x)-A_{p_2}(x)$. Moreover, the permutation in box $A$ must avoid the pattern
\raisebox{0.6ex}{\begin{tikzpicture}[scale = 0.3, baseline=(current bounding box.center)]
 			\foreach \x/\y in {0/0,0/1,1/0,1/2,2/1,2/2}		
 			\fill[gray!20] (\x,\y) rectangle +(1,1);
 			\draw (0.01,0.01) grid (2+0.99,2+0.99);
 			\node  at (1.6,1.6) {\tiny $p_1$};
 			\filldraw (1,1) circle (4pt) node[above left] {};
            \filldraw (2,2) circle (4pt) node[above left] {};
 			\end{tikzpicture}},
which is counted by $A(x)$, since $a$ is the leftmost possible element in an occurrence of $p$. Finally, $a$ and $b$ contribute the factor of $x^2$. This completes our proof of   \eqref{av-pattern-28-2} and hence gives the formula of $A_p(x)$ as desired.
\end{proof}

\begin{remark} Note that exactly the same enumeration result as that in Theorem~\ref{th:pattern 28-2} holds for the pattern \raisebox{0.6ex}{\begin{tikzpicture}[scale=0.3, baseline=(current bounding box.center)]
	\foreach \x/\y in {0/0,0/1,1/0,1/2,2/1,2/2}		
	\fill[gray!20] (\x,\y) rectangle +(1,1);
	\draw (0.01,0.01) grid (2+0.99,2+0.99);
	\filldraw (1,1) circle (4pt) node[above left] {};
	\filldraw (2,2) circle (4pt) node[above right] {};
	\node at (1.5,1.5) {\tiny $p_1$};
	\node at (0.5,2.5) {\tiny $p_2$};
	\end{tikzpicture}}.
Indeed, in the proof of Theorem~\ref{th:pattern 28-2} one just essentially need to substitute the ``leftmost $a$'' by the ``rightmost $b$''.\end{remark}

\subsection{Pattern Nr. 30}
We next consider generalizations of the pattern
Nr. 30 = $\pattern{scale=0.6}{2}{1/1,2/2}{0/1,0/2,1/0,1/1,1/2,2/0,2/1}$.

\begin{lem}[{\cite[Theorem 3.5]{SZ}}]\label{lem-pattern-30}
Let $p=\pattern{scale=0.6}{2}{1/1,2/2}{0/1,1/2,2/0,1/0,1/1,2/1,0/2}$,
	$F(x,q) =\sum_{n\geq 0}x^n\sum_{\pi\in S_n}q^{p(\pi)}$, and $A(x)$ be the generating function for $S(p)$. Then,
	\begin{align*}
	A(x)=\frac{(1+x)F(x)}{1+x+x^2F(x)},
	\ \ \ \ \
	F(x,q)=\frac{(1+x-qx)F(x)}{1+(1-q)x+(1-q)x^2F(x)}.
	\end{align*}
\end{lem}

The avoidance of the patterns in the next two theorems is based on Lemma~\ref{lem-pattern-30}, but we were not able to derive the distribution of these patterns extending the respective formula in Lemma~\ref{lem-pattern-30}.

\begin{thm}\label{th:pattern 30}
Suppose that $p$ is the pattern shown in Figure~\ref{pic-thm-pat-30}, where $p_1$ is a mesh pattern with the leftmost bottom box non-shaded, and the labels $a$, $b$, and $A$ are to be ignored. Then, the avoidance of $p$ is given by
\begin{align*}		
A_p(x)=  \frac{(1+x) F(x) + x^2 A_{p_1}(x)}{1+x+x^2F(x)}.
\end{align*}
\end{thm}

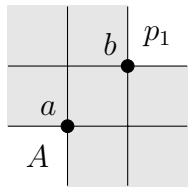
\begin{figure}[!ht]
\begin{center}
	\begin{tikzpicture}[scale=0.8, baseline=(current bounding box.center)]
	\foreach \x/\y in {0/1,0/2,1/0,1/1,1/2,2/0,2/1}		
	\fill[gray!20] (\x,\y) rectangle +(1,1);
	\draw (0.01,0.01) grid (2+0.99,2+0.99);
	\filldraw (1,1) circle (3pt) node[above left] {$a$};
	\filldraw (2,2) circle (3pt) node[above left] {$b$};
	\node at (2.5,2.5) {$p_1$};
    \node at (0.5,0.5) {$A$};
	\end{tikzpicture}
\caption{Related to the proof of Theorem~\ref{th:pattern 30}}\label{pic-thm-pat-30}
\end{center}
\end{figure}

\begin{proof}
We have the following functional equation
\begin{equation}\label{av-pattern-30}
A_p(x)+ \frac{x^2 F(x)}{1+x+x^2F(x)}\big(F(x)-A_{p_1}(x) \big)=F(x).
\end{equation}
Indeed, the right hand side counts all permutations, and the left hand side considers separately permutations avoiding $p$, counted by the $A_p(x)$ term in \eqref{av-pattern-30}, and those containing at least one occurrence of $p$. In the latter case, among all such occurrences, we can pick the occurrence $ab$ with the \emph{leftmost} possible $a$ as shown in Figure~\ref{pic-thm-pat-30}. Referring to this figure, we note that the North East box must contain at least one occurrence of $p_1$,  counted by $F(x)-A_{p_1}(x)$. Moreover, a permutation in box $A$ must avoid  both patterns
$\pattern{scale=0.6}{2}{1/1,2/2}{0/1,1/2,2/0,1/0,1/1,2/1,0/2}$ and
$\pattern{scale = 0.8}{1}{1/1}{0/1,1/0,1/1}$, since $a$ is the leftmost possible.
We denote by $B(x)$ the generating function of such permutations. Then, we have that
\begin{align*}
	B(x) + x B(x) = \frac{(1+x)F(x)}{1+x+x^2F(x)}
\end{align*}
by dividing the $\pattern{scale=0.6}{2}{1/1,2/2}{0/1,1/2,2/0,1/0,1/1,2/1,0/2}$-avoiding permutations, whose enumeration is given by Lemma~\ref{lem-pattern-30}, into two parts depending on whether they  avoid $\pattern{scale = 0.8}{1}{1/1}{0/1,1/0,1/1}$.
Note that when a permutation contains the pattern $\pattern{scale = 0.8}{1}{1/1}{0/1,1/0,1/1}$, the sub-permutation consisting of the first $n-1$ positions avoids both  patterns
$\pattern{scale=0.6}{2}{1/1,2/2}{0/1,1/2,2/0,1/0,1/1,2/1,0/2}$ and
$\pattern{scale = 0.8}{1}{1/1}{0/1,1/0,1/1}$.
Therefore, we get that
\begin{align*}
 B(x) = \frac{F(x)}{1+x+x^2F(x)}.
\end{align*}
Finally, $a$ and $b$ contribute the factor of $x^2$.
Substituting  the formula of $B(x)$ into  \eqref{av-pattern-30}, we obtain the desired formula of $A_p(x)$.
 This completes the proof.
\end{proof}

\subsection{Pattern Nr. 27}

    We next consider avoidance of a generalization of the pattern Nr. 27 = $\pattern{scale=0.6}{2}{1/1,2/2}{0/1,0/2,1/0,1/1,2/0,2/2}$.


\begin{thm}\label{th:pattern 27}
Suppose that $p$ is the pattern shown in Figure~\ref{pic-thm-pat-27}, where $p_1$ and $p_2$ are  any mesh patterns, and the labels $a$, $b$, and $A$ are to be ignored. Then, the avoidance of $p$ is given by
\begin{align*}		
A_p(x)=F(x)-x^2 \frac{F(x)}{1+xF(x)} \big(F(x)-A_{p_1}(x)\big) \big(F(x)-A_{p_2}(x) \big).
\end{align*}
\end{thm}

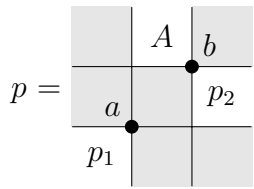
\begin{figure}[!ht]
\begin{center}
$p$~=~\begin{tikzpicture}[scale=0.8, baseline=(current bounding box.center)]
	\foreach \x/\y in {0/1,0/2,1/0,1/1,2/0,2/2}		
	\fill[gray!20] (\x,\y) rectangle +(1,1);
	\draw (0.01,0.01) grid (2+0.99,2+0.99);
	\filldraw (1,1) circle (3pt) node[above left] {$a$};
	\filldraw (2,2) circle (3pt) node[above right] {$b$};
	\node at (0.5,0.5) {$p_1$};
	\node at (2.5,1.5) {$p_2$};
	\node at (1.5,2.5) {$A$};
	\end{tikzpicture}
\caption{Related to the proof of Theorem~\ref{th:pattern 27}}\label{pic-thm-pat-27}
\end{center}
\end{figure}

\begin{proof}
We have the following functional equation:
\begin{equation}\label{av-pattern-27}
A_p(x)+x^2 B(x) \big(F(x)-A_{p_1}(x)\big) \big(F(x)-A_{p_2}(x) \big)=F(x),
\end{equation}
where $B(x)$ is the generating function for the number of $X$-avoiding permutations given in Theorem~\ref{thm-length-1}, which satisfies
$$B(x)=\frac{F(x)}{1+xF(x)}.$$
Indeed, $F(x)$ on the right hand side of \eqref{av-pattern-27} counts all permutations. On the left hand side of \eqref{av-pattern-27}, we count separately permutations avoiding $p$, counted  by the $A_p(x)$ term, and those containing at least one occurrence of $p$. In the latter case, among all occurrences of $p$, we pick the occurrence $ab$ with  the {\em leftmost} possible $b$  as shown in Figure~\ref{pic-thm-pat-27} and $a$ is then uniquely determined. Referring to this figure, we note that the South West box must contain at least one occurrence of $p_1$, counted by $F(x)-A_{p_1}(x)$, and the East box must contain at least one occurrence of $p_2$, counted by $F(x)-A_{p_2}(x)$. Moreover, the permutation in the box $A$ must avoid the pattern $X$, counted by $B(x)$, since $b$ is the leftmost possible. Finally, $a$ and $b$ contribute the factor of $x^2$. Thus, this completes the proof of~\eqref{av-pattern-27}.
Substituting the formula of $B(x)$ into~\eqref{av-pattern-27}, we obtain the desired formula of   $A_p(x)$.
\end{proof}

Theorem~\ref{th:pattern 27} generalizes  the avoidance  of the pattern Nr.\ 27. However, generalizing its distribution is hard, because we need to control at the same time occurrences of the patterns $p$ and $Z = \pattern{scale = 0.8}{1}{1/1}{0/0,1/1}$ in the box $A$.
Moreover, we cannot further generalize Theorem~\ref{th:pattern 27} by placing a mesh pattern $p_3$ in the box $A$, because  $A$ must avoid $Z$ when requiring from $b$ to be the leftmost, so we will be forced to control two patterns $p_3$ and $Z$ at the same time. Of course, we can require from $b$ to be the rightmost, but then we will be forced to control two patterns in the East box. Finally, swapping $A$ and $p_2$ in Theorem~\ref{th:pattern 27} leads to the same enumeration result, which is not hard to see.

\subsection{Pattern Nr. 34}
We next consider generalizations of the pattern
Nr. 34 = $\pattern{scale=0.6}{2}{1/1,2/2}{0/0,0/1,1/0,1/1,1/2,2/1,2/2}$.

\begin{lem}[{\cite[Theorem 3.7]{SZ}}]\label{lem-pat-34}
	Let $p=\pattern{scale=0.6}{2}{1/1,2/2}{0/1,1/2,0/0,2/2,1/0,1/1,2/1}$.
	Then, the avoidance and distribution of $p$ are
	\begin{align*}
	A_p(x)= \frac{F(x)}{1+x^2F(x)},
	\ \ \ \ \
	F_p(x,q)=\frac{F(x)}{1+(1-q)x^2F(x)}.
	\end{align*}
\end{lem}


Replacing the two elements in the pattern  Nr.\ 34 by the pattern $1p_1$, where $p_1$ is any permutation of $\{2,3,\ldots, k\}$, $k\geq 2$, with all boxes shaded as in Figure~\ref{pic-thm-pat-34}, we can apply essentially the same arguments as in the proof of Lemma~\ref{lem-pat-34} in \cite{SZ} to obtain the following theorem.


\begin{figure}[!ht]
\begin{center}
$p$~=~\begin{tikzpicture}[scale=0.8, baseline=(current bounding box.center)]
	\foreach \x/\y in {0/0,0/1,1/0,1/1,1/2,2/1,2/2}		
	\fill[gray!20] (\x,\y) rectangle +(1,1);
	\draw (0.01,0.01) grid (2+0.99,2+0.99);
	\filldraw (1,1) circle (3pt) node[above left] {};
    \node at (1.5,1.5) {$p_1$};
	\end{tikzpicture}
\caption{Related to Theorem~\ref{thm-pat-34}}\label{pic-thm-pat-34}
\end{center}
\end{figure}
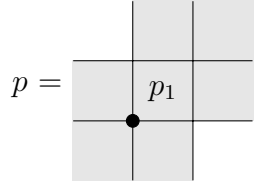

\begin{thm}\label{thm-pat-34}
	Suppose that $p$ is the pattern shown in Figure~\ref{pic-thm-pat-34}, where $k\geq 1$ elements are in increasing order in the middle box. Then, the avoidance and distribution of $p$ are given by
	\begin{align*}
	A_p(x)= \frac{F(x)}{1+x^kF(x)},
	\ \ \ \ \
	F(x,q)=\frac{F(x)}{1+(1-q)x^kF(x)}.
	\end{align*}
\end{thm}

    The avoidance of the pattern Nr.\ 34 can  be generalized, which is done in the next theorem,  but the distribution is hard because we need to control $p_2$ and $p$ in the same box in that theorem.

\begin{thm}\label{th:pattern 34-2}
Suppose that $p$ is the pattern shown in Figure~\ref{pic-thm-pat-34-2}, where $p_1$ is any permutation of $\{2,3,\ldots, k\}$, $k\geq 1$, with all boxes shaded, $p_2$ is any mesh pattern, and the labels $a$ and $A$ are to be ignored. Then, the avoidance of $p$ is given by
\begin{align*}		
A_p(x)=F(x)-  \frac{x^k F(x)}{1+x^k F(x)} \big(F(x)-A_{p_2}(x)\big).
\end{align*}
\end{thm}

\begin{figure}[!ht]
\begin{center}
$p$~=~\begin{tikzpicture}[scale=0.8, baseline=(current bounding box.center)]
	\foreach \x/\y in {0/0,0/1,1/0,1/1,1/2,2/1,2/2}		
	\fill[gray!20] (\x,\y) rectangle +(1,1);
	\draw (0.01,0.01) grid (2+0.99,2+0.99);
	\filldraw (1,1) circle (3pt) node[above left] {};
    \node at (1.5,1.5) {$p_1$};
	\node at (2.5,0.5) {$p_2$};
	\filldraw (1,1) circle (3pt) node[above left] {$a$};
	\node at (0.5,2.5) {$A$};
	\end{tikzpicture}
\caption{Related to the proof of Theorem~\ref{th:pattern 34-2}}\label{pic-thm-pat-34-2}
\end{center}
\end{figure}
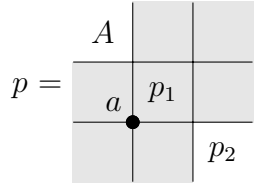

\begin{proof}
Let $D(x)$ be the generating function for the number of permutations avoiding the pattern in Figure~\ref{pic-thm-pat-34}.
Then, it follows from Theorem~\ref{thm-pat-34} that
$$D(x)=\frac{F(x)}{1+x^kF(x)}.$$
We have the following functional equation:
\begin{equation}\label{av-pattern-34}
A_p(x)+x^k D(x)  \big(F(x)-A_{p_2}(x) \big)=F(x).
\end{equation}
Indeed, the right hand side counts all permutations. On the left hand side, we count separately permutations avoiding $p$, counted by $A_p(x)$, and those containing at least one occurrence of $p$. In the latter case, among all occurrences of $p$, we can pick the occurrence with the \emph{leftmost} possible $a$ as shown in Figure~\ref{pic-thm-pat-34-2}. Referring to this figure, we note that the South East box must contain at least one occurrence of $p_2$, counted by $F(x)-A_{p_2}(x)$. Moreover, the permutation in box $A$ must avoid the pattern in Figure~\ref{pic-thm-pat-34}
 that is counted by $D(x)$, since  $a$ is the leftmost possible. There are no other restrictions on $A$ because the pattern in Figure~\ref{pic-thm-pat-34} cannot begin in $A$ and end somewhere else. Finally, the $k$ elements in the middle box contribute the factor of $x^k$. Thus, by combing with  the formula of $D(x)$, we complete the proof of   \eqref{av-pattern-34},  and  hence give the formula of $A_p(x)$.
\end{proof}

\section{Concluding remarks}\label{final-sec}

We have a number of general results related to distribution or avoidance of several infinite families of mesh patterns. How to describe the class of mesh patterns for which our distribution or avoidance results can be applicable? Namely, in which situations one can break the problem of enumerating mesh patterns into smaller problems using our theorems? What is the complexity of recognizing the class?

\section*{\bf Acknowledgments}
The first author is grateful to the administration of the Center for Combinatorics at Nankai University for their hospitality during the author's stay in April 2018.
The second author was partially supported by the National Science Foundation of China (Nos.  11701424).

\end{document}